\documentclass[12pt]{amsart}

\usepackage{amsmath}
\usepackage{hyperref}
\usepackage{graphicx}
\usepackage{tikz-cd}

\theoremstyle{plain}
\newtheorem{thm}{Theorem}[section]
\newtheorem{cor}[thm]{Corollary}
\newtheorem{prop}[thm]{Proposition}
\newtheorem{lem}[thm]{Lemma}
\theoremstyle{definition}

\theoremstyle{remark}
\newtheorem{rem}[thm]{Remark}

\numberwithin{equation}{subsection}

\renewcommand{\bold}[1]{\medskip \noindent {\bf #1 }\nopagebreak}

\newcommand{\nc}{\newcommand}
\newcommand{\rnc}{\renewcommand}
\newcommand{\e}{\varepsilon}

\nc\bA{\mathbb{A}}
\nc\bB{\mathbb{B}}
\nc\bC{\mathbb{C}}
\nc\bD{\mathbb{D}}
\nc\bE{\mathbb{E}}
\nc\bF{\mathbb{F}}
\nc\bG{\mathbb{G}}
\nc\bH{\mathbb{H}}
\nc\bI{\mathbb{I}}
\nc{\bJ}{\mathbb{J}} 
\nc\bK{\mathbb{K}}
\nc\bL{\mathbb{L}}
\nc\bM{\mathbb{M}}
\nc\bN{\mathbb{N}}
\nc\bO{\mathbb{O}}
\nc\bP{\mathbb{P}}
\nc\bQ{\mathbb{Q}}
\nc\bR{\mathbb{R}}
\nc\bS{\mathbb{S}}
\nc\bT{\mathbb{T}}
\nc\bU{\mathbb{U}}
\nc\bV{\mathbb{V}}
\nc\bW{\mathbb{W}}
\nc\bY{\mathbb{Y}}
\nc\bX{\mathbb{X}}
\nc\bZ{\mathbb{Z}}
\nc\cA{\mathcal{A}}
\nc\cB{\mathcal{B}}
\nc\cC{\mathcal{C}}
\nc\cD{\mathcal{D}}
\nc\cE{\mathcal{E}}
\nc\cF{\mathcal{F}}
\nc\cG{\mathcal{G}}
\nc\cH{\mathcal{H}}
\nc\cI{\mathcal{I}}
\nc{\cJ}{\mathcal{J}} 
\nc\cK{\mathcal{K}}
\nc\cL{\mathcal{L}}
\nc\cM{\mathcal{M}}
\nc\cN{\mathcal{N}}
\nc\cO{\mathcal{O}}
\nc\cP{\mathcal{P}}
\nc\cQ{\mathcal{Q}}
\nc\cR{\mathcal{R}}
\nc\cS{\mathcal{S}}
\nc\cT{\mathcal{T}}
\nc\cU{\mathcal{U}}
\nc\cV{\mathcal{V}}
\nc\cW{\mathcal{W}}
\nc\cY{\mathcal{Y}}
\nc\cX{\mathcal{X}}
\nc\cZ{\mathcal{Z}}

\nc{\dmo}{\DeclareMathOperator}
\rnc{\Re}{\operatorname{Re}}
\rnc{\Im}{\operatorname{Im}}
\dmo{\Isom}{Isom}

\newcommand{\dzero}[1]{\left.\frac{d#1}{dt}\right|_{t=0}}
\newcommand{\dbar}{\overline\partial}

\newcommand{\norm}[1]{\left\Vert #1 \right\Vert}
\newcommand{\abs}[1]{\left| #1 \right|}

\let\from\colon
\let\C\bC
\dmo{\Teich}{Teich}
\let\teich\Teich
\dmo{\hodge}{Hodge}

\nc{\delz}{\frac{\partial}{\partial z}}
\nc{\delbar}{\overline{\partial}}

\newcommand{\dzbar}{\overline{dz}}

\title{Hodge and Teichm\"uller}

\author[Kahn]{Jeremy~Kahn}
\author[Wright]{Alex~Wright}

\begin{document}

\maketitle
\thispagestyle{empty}

\begin{abstract}
We consider the derivative $D\pi$ of the projection $\pi$ from a stratum of Abelian or quadratic differentials to Teichm\"uller space. A closed one-form $\eta$ determines a relative cohomology class $[\eta]_\Sigma$, which is a tangent vector to the stratum. We give an integral formula for the pairing of of $D\pi([\eta]_\Sigma)$ with a cotangent vector to Teichm\"uller space (a quadratic differential). 

We derive from this a comparison between Hodge and Teichm\"uller norms, which has been used in the work of Arana-Herrera on effective dynamics of mapping class groups, and which may clarify the relationship between dynamical and geometric hyperbolicity results in Teichm\"uller theory.   
\end{abstract}

\section{Introduction}

\bold{The derivative of the projection.} Each stratum $\cH(\kappa)$ of genus $g$ Abelian differentials comes equipped with a projection map $\pi\from \cH(\kappa)\to \cM_g$ to the moduli space of Riemann surfaces, defined by $\pi(X,\omega)=X$. 

Given $(X, \omega)\in \cH(\kappa)$, we let $\Sigma=\{z_1, \ldots, z_s\}$ denote the set of zeros of $\omega$. Any complex valued closed differential one-form $\eta$ on $X$ determines a relative cohomology class $[\eta]_\Sigma \in H^1(X, \Sigma, \bC)$. Since $H^1(X, \Sigma, \bC)$ is the tangent space to  $\cH(\kappa)$, we think of $[\eta]_\Sigma$ as a tangent vector to the stratum. 

Using a point-wise decomposition, any such $\eta$ can be written uniquely as $\eta=\eta^{1,0}+\eta^{0,1}$, where $\eta^{1,0}$ and $\overline{\eta^{0,1}}$ are of type $(1,0)$ and need not be closed. 
The cotangent space to $X$ is the space $Q(X)$ of quadratic differentials on $X$, so a tangent vector to $\cM_g$ can be thought of as a linear functional on $Q(X)$.  

\begin{thm}\label{T:closed}
Let $\eta=\eta^{1,0}+\eta^{0,1}$ be closed as above. 
Let $\{\gamma_1, \ldots, \gamma_s\}$ denote small disjoint positively oriented loops around the zeros of $\omega$, and let $X'$ be the complement in $X$ of the discs that they bound. Then the pairing of $(D\pi)[\eta]_\Sigma$ with a quadratic differential $q$ is equal to 
\begin{equation} \label{eq:main}
\int_{X'} q \frac{\eta^{0,1}}{\omega} +  \frac1{2i}\sum_j\int_{\gamma_j} \frac{F_j}\omega q,
\end{equation}
where $F_j(z) = \int_{z_j}^z \eta$ is defined by integrating along paths in the disc containing $z_j$. 
\end{thm}

%
%

See the beginning of Section \ref{sec:eq} for how the tensors in this formula should be interpreted. The fact that equation \ref{eq:main} does not depend on the choice of loops $\gamma_j$ follows from Stokes' Theorem, as in the proof of Lemma \ref{L:stokes}.

 In Corollary \ref{C:harmonic}, we show that if $\eta$ is harmonic, one can let the size of the loops go to zero to obtain a formula in terms of residues. 
In Remark \ref{R:marked2}, we note that Theorem \ref{T:closed} holds also for the projection to $\cM_{g,s}$ obtained by marking all the zeros of $\omega$. In Section \ref{SS:QD}, we explain how to apply Theorem \ref{T:closed} to strata of quadratic differentials.

\bold{The principal stratum of quadratic differentials.} Any quadratic differential $(X,q)$ admits a canonical double cover $\rho_q\from\hat{X} \to X$  on which the pullback of $q$ becomes the square of an Abelian differential $\omega$. The Deck group of this cover is an involution $\tau$, and $\tau^*(\omega)=-\omega$. The cover is sometimes called the holonomy double cover, and $\omega$ is sometimes called the square-root of $q$. 

Let $\Sigma$ continue to denote the set of zeros of $\omega$. Denote by $H^1_{-1}(\hat{X}, \Sigma, \bC)$  and $H^1_{-1}(\hat{X}, \bC)$ the $-1$ eigenspace for the action of $\tau$ on $H^1(\hat{X}, \Sigma, \bC)$  and $H^1(\hat{X}, \bC)$ respectively. 

The tangent space to the stratum of $(X,q)$ at the point $(X,q)$ is $H^1_{-1}(\hat{X}, \Sigma, \bC)$. If $q$ has no even order zeros, then the natural map  $H^1_{-1}(\hat{X}, \Sigma, \bC)\to H^1_{-1}(\hat{X}, \bC)$ is an isomorphism, and so $H^1_{-1}(\hat{X}, \bC)$ can also be viewed as the tangent space. Every element of absolute cohomology can be represented uniquely by a harmonic one-form $\eta$, so an arbitrary element of the tangent space uniquely corresponds to an anti-invariant form $\eta$ on $\hat{X}$ with $\eta^{1,0}\in H^{1,0}(\hat{X})$ and $\eta^{0,1}\in H^{0,1}(\hat{X})$.

Since $\omega$ and $\eta^{0,1}$ are both in the $-1$ eigenspace of $\tau$, the Beltrami differential $\eta^{0,1}/\omega$ is $\tau$ invariant and hence is the pull back of a Beltrami differential on $X$, which we will continue to denote $\eta^{0,1}/\omega$. 

When $(X,q)$ is in the principal stratum, Corollary \ref{C:harmonic} further simplifies to give the following, where $\pi$ denotes the projection from the principal stratum of quadratic differentials to the moduli space of Riemann surfaces. 

\begin{cor}\label{C:principal}
If $(X,q)$ is in the principal stratum, and  $\eta$ is a harmonic anti-invariant form on $\hat{X}$, then the pairing of $(D\pi)[\eta]_\Sigma$ with a quadratic differential $q'\in Q(X)$ is equal to 
$$\frac12\int_{X} q'\frac{\eta^{0,1}}{\omega}.$$
\end{cor} 

We warn that the Beltrami differential $\frac{\eta^{0,1}}{\omega}$ is not bounded and hence does not define a tangent vector to $\cM_g$ in the usual way. The integrand $q' \frac{\eta^{0,1}}{\omega}$ is integrable (in fact its pullback to $\hat{X}$ is continuous), so $\frac{\eta^{0,1}}{\omega}$ defines a functional on the cotangent space $Q(X)$ via integration, and hence defines a tangent vector indirectly in this way. 

\begin{rem}\label{R:fiber}
Corollary \ref{C:principal} in particular witnesses that the tangent space to the fiber of $\pi$ is $H^{1,0}_{-1}(\hat{X})$. In fact, if $q'\in Q(X)$, it is not hard to see that $\rho_q^*(q')/(2\omega)$ is contained in  $H^{1,0}_{-1}(\hat{X})$ and moreover is the derivative of the path $(X, q+tq')$ at $t=0$ \cite{DH}. 
\end{rem}

Keeping in mind that the kernel of $D\pi$ is $H^{1,0}_{-1}(\hat{X})$, we consider $\eta\in H^{0,1}_{-1}(\hat{X})$, and compare the Hodge norm of $\eta$ and the Teichm\"uller norm of $D\pi(\eta)$. 

\begin{thm}\label{T:HvT}
Assuming $\omega$ has area 1, we have 
$$\norm{\eta}_{\hodge} \le \norm{D\pi(\eta)}_{\Teich} \le \frac{4}{r}  \norm{\eta}_{\hodge},$$
for any $\eta\in H^{0,1}_{-1}(\hat{X})$, where $2r$ is the length of the shortest saddle connection on $(\hat{X},\omega)$. 
\end{thm}

This is somewhat reminiscent of comparisons between the Teichm\"uller and Weil-Petterson norms in \cite[Lemma 5.4]{BMW}.

\bold{Other perspectives and previous results.} One can of course obtain formulas for $D\pi$ by, for example,  triangulating the surface and considering Beltrami differentials of maps that are affine on each triangle; or picking an open cover and using Cech cohomology, as in \cite{HM}. McMullen gave a formula in terms of complex twists  \cite{McM}. Derivatives  of some especially important deformations  have been given in \cite{W}.

\bold{Motivation and significance.} We  had two  specific motivations for writing this paper, both having to do with Theorem \ref{T:HvT}. 
\begin{enumerate}
\item Arana-Herrera has used Theorem \ref{T:HvT} in his proof of an effective version of the lattice point counting problem of Athreya-Bufetov-Eskin-Mirzakhani \cite{AH1, ABEM}. This in turn is used in his subsequent work on the effective dynamics of the mapping class group \cite{AH2, AH3}. 
\item Since the work of Forni, the hyperbolicity of the Teichm\"uller geodesic flow has been studied using the Hodge norm \cite{Forni}, see also \cite{FM} for a survey, \cite{EMR, F} for more recent developments, and \cite{Kont} for the introduction of the Hodge norm to Teichm\"uller dynamics by Kontsevich. On the other hand, geometric hyperbolicity results are expressed in terms of Teichm\"uller distance and are proven using very different techniques \cite{Rafi}. Theorem \ref{T:HvT} opens the door to links between these dynamical and geometric results, including the possibility of (re)proving geometric hyperbolicity results using dynamical hyperbolicity. 
\end{enumerate}  

It is also conceivable that Theorem \ref{T:closed} could be useful in the study of $GL(2,\bR)$ orbit closures of translation surfaces, using the restrictions on the period mapping exploited in \cite{MW}.

Corollary \ref{C:principal} can  be used to re-derive the fact that the canonical symplectic form on the principal stratum, obtained as for the cotangent bundle to any manifold, corresponds to the usual symplectic form  on $H^1_{-1}(\hat{X}, \bC)$ \cite{BKN}. This is notable in part because Teichm\"uller geodesic flow is easily seen to be Hamiltonian using the later symplectic form \cite{Masur}.

\bold{A tangential remark.} The following requires only Remark \ref{R:fiber}, but illustrates another potential connection between this paper and other recent work. 

\begin{prop}\label{P:TotGeodesic}
Let $\cM$ be a $GL(2, \bR)$ orbit closure in the principal stratum of quadratic differentials over $\cM_{g,n}$, and let $\pi$ denote the projection to $\cM_{g,n}$. Then $\overline{\pi(\cM)}$ is a totally geodesic subvariety of $\cM_{g,n}$.
\end{prop}

\begin{proof}
By \cite{EM, EMM} each orbit closure is a properly immersed smooth suborbifold, and by \cite{Fi1, Fi2} it is moreover an algebraic variety. 

When $\cM$ is in the principal stratum, its tangent space at $(X,q)\in \cM$ is naturally identified with a  subspace of $H^{1}_{-1}(\hat{X},\bC)$ defined over $\bR$. By \cite{Fi1}, this subspace is the direct sum of its intersections with $H^{1,0}_{-1}(\hat{X})$ and $H^{0,1}_{-1}(\hat{X})$. Since these two intersections are complex conjugate, they have the same dimension, which witnesses the fact that $\cM$ must have even complex dimension. Since $H^{1,0}_{-1}(\hat{X})$ is the kernel of $D\pi$, we see that the kernel of $D\pi$ restricted to $\cM$ has dimension half that of $\cM$ at every point. 

We conclude that the variety $\pi(\cM)$ has dimension half that of $\cM$. It follows from previous observations that its closure is totally geodesic, as is implicit in \cite{MMW, EMMW, WTG} and explicit in \cite[Proposition 1.3]{GS}. 
\end{proof}

The only known non-trivial primitive totally geodesic subvarieties of dimension greater than 1 have dimension 2 and arise from orbit closures in the principal strata of $\cM_{1,3}$, $\cM_{1,4}$ and $\cM_{2,1}$ \cite{MMW, EMMW}. See also \cite{WTG} for related results, and \cite{GS} for a survey.  

Before \cite{MMW}, it was not expected that any orbit closures would give rise to totally geodesic subvarieties, and \cite{MMW, EMMW} used a detailed understanding to conclude that three orbit closures have sufficiently small projections to give totally geodesic surfaces. So it is perhaps surprising that Proposition \ref{P:TotGeodesic} indicates that, at least for orbit closures in the principal stratum, there is automatically an associated totally geodesic subvariety. 

 
\bold{Acknowledgements.} 
During the preparation of this paper, the first author was partially supported by the Simons Foundation and the second author was partially supported by a Clay Research Fellowship,  NSF Grant DMS 1856155, and a Sloan Research Fellowship. We thank Francisco Arana-Herrera, Giovanni Forni, and Bradley Zykoski for comments, and especially thank Curt McMullen for comments that lead to a  simplification of the proof of Theorem \ref{T:closed} and the recovery of a previously missing factor of $2i$ (see \eqref{eq:ridiculous}). McMullen also informed us that Theorem \ref{T:closed} can also be derived from \cite[Theorem 1.2]{McM}.

The title of this paper is inspired by \cite{JJ}.

\section{The derivative formula}\label{sec:eq}

Before we give the proofs of Theorem \ref{T:closed} and Corollary \ref{C:harmonic}, we must clarify our conventions. 
In the integral over $X'$ in \eqref{eq:main},
the integrand, as a product of a $(2, 0)$, a $(-1, 0)$, and a $(0, 1)$ form, is a $(1, 1)$ form, 
and hence can be written locally on $U \subset \C$ as $f\, dz\, \dzbar$. 
We then define, as a matter of convention, 
\begin{equation} \label{eq:one-one}
\int_U f\, dz\, \dzbar = \int_U f\, dx\, dy;
\end{equation}
we observe that the latter integral will be independent of the choice of coordinate and that \eqref{eq:one-one} can be used to define a global integral in the usual way. 
This is the convention used in defining the Teichm\"uller pairing of a quadratic differential $q$ and a Beltrami differential $\mu$ as $\int_X q\mu$.
Since $\dzbar \wedge dz = 2i \,dx\, dy$,
\begin{equation} \label{eq:ridiculous}
\int_U f \, dz\, \dzbar = \frac1{2i}\int_U f\, \dzbar \wedge dz.
\end{equation}

We start the proof of Theorem \ref{T:closed} with the following observation. 

\begin{lem}\label{L:stokes}
If $[\eta]_\Sigma=0$, then 
$$\int_{X'} q \frac{\eta^{0,1}}{\omega} +  \frac1{2i}\sum\int_{\gamma_j} \frac{F_j}\omega q=0.$$
\end{lem}

Thus,  the derivative formula \eqref{eq:main} only depends on the cohomology class of $\eta$.

\begin{proof}
$[\eta]_\Sigma=0$ implies that $\eta=df$ for a function $f$ which is zero on $\Sigma$. We note 
$$
d\left(\frac{f}\omega q\right) = \delbar \left(\frac{f}\omega q\right) = \eta^{0,1}\wedge\frac q\omega,
$$
and hence, by \eqref{eq:ridiculous} and Stokes' Theorem,
$$
2i \int_{X'} q \frac{\eta^{0,1}}{\omega}
= \int_{X'} \eta^{0, 1} \wedge \frac q\omega
= - \sum_j \int_{\gamma_j} \frac{f}\omega q.
$$
Since both $f$ and $F_j$ are zero at $z_j$ and have exterior derivative $\eta$, we see that $F_j$ is the restriction of $f$, so this gives the result. 
\end{proof}

Next, consider the Beltrami differential 
$$\mu_t = \frac{t\eta^{0,1}}{\omega+t\eta^{(1,0)}}$$
on $X$. Assuming $\eta$ is compactly supported on $X-\Sigma$ and $t$ is sufficiently small, then $\|\mu_t\|_\infty <1$. In this case, let $X_t$ denote $X$ with the complex structure for which the identity map $X\to X_t$ has Beltrami differential $\mu_t$.

\begin{lem}\label{L:newC}
If $\eta$ is compactly supported and $t$ is sufficiently small, then the closed one-form $\omega+t\eta$ is holomorphic
on $X_t$. 
\end{lem}

\begin{proof}
Write  $\eta^{1,0}=\eta^{1,0}(z)dz$ etc. 
Since $\omega+t\eta$ is proportional to 
$$dz + \frac{t\eta^{0,1}(z)}{\omega(z)+t\eta_{1,0}(z)} d\overline{z},$$
we see that  $\omega+t\eta$ is $(1,0)$ form on $X_t$. Since closed $(1,0)$ forms are holomorphic, this gives the result. 
\end{proof}

\begin{proof}[Proof of Theorem \ref{T:closed}]
By Lemma \ref{L:stokes}, we can assume that $\eta$ is compactly supported on $X'$. Lemma \ref{L:newC} gives that $(X_t, \omega+t\eta)$ is a path in the stratum whose image in period coordinates is $[\omega+t\eta]_\Sigma$.  Since $\dzero{\mu_t} = \eta^{0,1}/\omega$, the derivative of the family $X_t$ is the Beltrami differential $\eta^{0,1}/\omega$. Since $\eta$ is is supported in $X'$, the pairing of this Beltrami differential with $q$ coincides with the formula given in Theorem \ref{T:closed}.
\end{proof}

We now observe the the formula also simplifies if $\eta$ is harmonic. 

\begin{cor}\label{C:harmonic}
Let $\eta$ be harmonic, i.e. $\eta^{1,0}\in H^{1,0}(X)$ and $\eta^{0,1}\in H^{0,1}(X)$. 
Then the pairing of $q$ and $[\eta]_\Sigma$ is given by 
$$\int_{X} q \frac{\eta^{0,1}}{\omega} +\pi \sum   \operatorname{res}_{z_j}\left( \frac{\int_{z_j}^z \eta^{1,0}}{\omega}q \right),$$
where the first term is understood as a Cauchy principal value using flat coordinates provided by $\omega$. 
\end{cor}

The Cauchy principal value is the limit as $\e\to 0$ of the integral over the subset of $X$ whose $\omega$-distance to a zero of $\omega$ is at least $\e$. 

\begin{proof}
It suffices to check that the limit of 
$$\int_{\gamma_j} \frac{\int_{z_j}^z \eta^{0,1}}\omega q$$
as the size of the loop $\gamma_j$ goes to zero is 0. We can pick local coordinates in which $\omega= z^k dz$, and pick $\gamma_j$ of the form $re^{i\theta}$. We expand $q=(\sum c_\ell z^\ell) dz^2$, and we expand $\int_{z_j}^z \eta^{0,1} = \sum_{m>0} d_m \overline{z}^m$, noting that (because of the integral) $m=0$ is not included in the index of this last sum. 
For each value of $\ell$ and $m$ we obtain a term proportional to $\int_{\gamma_j} z^{-k+\ell} \overline{z}^m dz$, which is proportional to 
$$  
\int_0^{2\pi} z^{-k+\ell+1} \overline{z}^m d\theta = r^{-k+\ell+1+m} \int_0^{2\pi} e^{i \theta (-k+\ell+1-m)} d\theta.$$
This is non-zero only when $-k+\ell+1-m=0$, in which case $r^{-k+\ell+1+m} = r^{2m}$. Since $m>0$, this integral goes to zero as $r\to 0$. 
\end{proof}

\begin{rem}\label{R:marked2}
All the results of this section also apply to the projection to $\cM_{g,s}$, with $s=|\Sigma|$, obtained by marking all the zeros of $\omega$. We can also allow $\Sigma$ to be a finite set which properly contains the zeros of $\omega$, in which case we think of a point of $\Sigma$ at which $\omega$ does not vanish as a zero of order zero.
\end{rem}

\newcommand{\betabar}{\overline{\beta}}
\section{The principal stratum}

\subsection{Arbitrary strata of quadratic differentials}\label{SS:QD} Let $\cQ(\kappa)$ denote a stratum of the bundle of quadratic differentials over $\cM_{g,n}$. Given $(X,q)\in \cQ(\kappa)$, we continue to let $\rho_q\from\hat{X} \to X$ denote the double cover on which the pullback of $q$ becomes the square of an Abelian differential $\omega$. We continue to specify a tangent vector to the strata by giving a closed anti-invariant one-form $\eta$ on $\hat{X}$, and we consider also a cotangent vector $q'\in QD(X)$ to $\cM_{g,n}$. 

As we now explain, if $\pi\from \cQ(\kappa)\to \cM_{g,n}$ continues to denote the projection, then the pairing of the cotangent vector $D\pi([\eta])$ with the cotangent vector $q'$ is
$$
\frac12\int_{\hat{X}'} \rho_q^*(q') \frac{\eta^{0,1}}{\omega} 
+  \frac12 \sum_i \int_{\gamma_i} \frac{F_i}\omega \rho_q^*(q'),
$$
where $F_i(z) = \int_{z_i}^z \eta$ is defined as before. This will follow from Theorem \ref{T:closed} once we clarify the bookkeeping of double covers. 

Let $U$ be a small open subset of $H^1_{-1}(\hat{X}, \bC)$ which contains $[\omega]$ and provides a local coordinate for $\cQ(\kappa)$ at $(X,q)$. We restrict $\pi$ to give a map from $U$ to  $\cM_{g,n}$. 

Let $\hat{g}$ denote the genus of $\hat{X}$. Assuming that  $U$ is simply connected, we can arbitrarily pick a marking to get a map $\hat{\pi}\from U \to \cG$, where $\cG \subset \cT_{\hat{g}}$ is the subset where there is an involution in the mapping class of the involution $\tau$ which negates $\omega$. The quotient by this involution gives a map $\rho \from \cG \to \cM_{g,n}$. We have $\pi = \rho\circ \hat{\pi}$. 

Given any differentiable map $\pi = \rho\circ \hat{\pi}$, and any cotangent vector $q'$ in the codomain and tangent vector $[\eta]$ in the domain, we have 
$$\langle D\pi(\eta), q'\rangle = \langle D\hat{\pi}(\eta), \rho^* (q')\rangle,$$
where $\langle \cdot, \cdot \rangle$ is the  pairing between tangent and cotangent vectors and the coderivative $\rho^*$ maps between cotangent spaces. 

In our situation, the cotangent space to $\cG$ is the space of quadratic differentials which are invariant by the involution, and if $q'$ is a quadratic differential on $X$ then
$$
\rho^* q' = \frac12 \rho_q^* q'.
$$
Let $\iota\from \cG\to \cT_{\hat{g}}$ denote the inclusion. Any invariant quadratic differential can be considered both as a cotangent vector to $\cG$ and to $\cT_{\hat{g}}$, and we have 
$\langle D\iota(v), q'' \rangle = \langle v, q'' \rangle$ for any tangent vector $v$ to $\cG$ and any invariant quadratic differential $q''$ on $\hat X$. 
Thus we get 
$$\langle D\pi(\eta), q'\rangle = \langle D(\iota \circ \hat{\pi})(\eta), \rho^* (q')\rangle.$$
This gives the desired formula, since $\iota \circ \hat{\pi}$ is the projection from the stratum $\cH$ of $(\hat{X}, \omega)$ to $\cT_{\hat{g}}$, whose derivative is given by Theorem \ref{T:closed}.  

Applying this to the principal stratum, we get the following. 

\begin{proof}[Proof of Corollary \ref{C:principal}.] 
First note that, since $q$ has simple zeros, $\omega$ has double zeros. Since each zero of $\omega$ is a ramification point of $\rho_q$, every pull back $\rho_q^*(q')$ has at least a double zero at every zero of $\omega$, so we see that $\rho_q^*(q')/\omega$ is holomorphic at the zeros of $\omega$. This shows that all the residue terms in Corollary \ref{C:harmonic} are zero, giving the result.
\end{proof}

\newcommand{\qq}{\rho_q^*(q')}
%
%
%


\subsection{Norm comparisons.}
Consider a tangent vector to a principal stratum of the form $\eta=\betabar$, where $\betabar \in H^{1,0}_{-1}(\hat{X}).$ Keeping in mind Corollary \ref{C:principal}, the Beltrami differential $\mu = \betabar/\omega$ can be viewed as the tangent vector $D\pi(\eta)$ via pairing with cotangent vectors. We now turn to  comparisons between the Hodge norm 
$$
\norm{\beta}_{\hodge} =\sqrt{ \int_{\hat X} |\beta|^2}
$$
and the Teichmuller norm
$$
\norm{[\mu]}_{\teich} 
= \sup_{\norm{q'} = 1} \int_X q' \mu
= \sup_{\norm{q'} = 1} \int_{\hat X}\frac12\qq\frac{\betabar}\omega. 
$$

\begin{thm}
We have
\begin{equation} \label{eq:basic-hodge-teich}
 \norm{[\mu]}_{\Teich}
 \ge
 \frac{\norm{\beta}_{\hodge}} {\norm{\omega}_{\hodge}}
\end{equation}
\end{thm}

The reader should keep in mind that the normalization $\norm{q} = 1$ corresponds to $\norm{\omega}_{\hodge} = \sqrt 2$. From now on we will omit the subscript ``Hodge", since the only norm we will consider for Abelian differentials is the Hodge norm.


\begin{proof}
Let
$
q' = (\rho_q)_* (\omega\beta),
$,
where $(\rho_q)_*$ is defined by summing over fibers,
so $\omega\beta = \frac12 \rho_q^* q'$. 
Then $\norm{q'} = \norm{\omega\beta} \le \norm{\omega}\norm{\beta}$  by the Cauchy-Schwartz inequality. 
On the other hand,
$$
\int_X q' \mu = \int_{\hat X} \omega\beta \frac{\betabar}\omega = \int_{\hat X} \beta\betabar = \norm{\beta}^2. 
$$
Therefore
$$
\norm{[\mu]}_{\Teich} \ge\frac{\norm{\beta}^2}{\norm{q'}}
\ge \frac{\norm{\beta}^2}{\norm{\beta}{\norm{\omega}}} = \frac{\norm{\beta}}{\norm{\omega}}. \qedhere
$$

%
\end{proof}

For the other direction it will be helpful to observe the following, for any holomorphic function $f\from D_r\to \C$,
where $D_r$ denotes the disk of radius $r$. 
\begin{align}
|f(0)| &\le \frac{\int_{D_r} |f(z)|}{\pi r^2} \\
&\le \sqrt{\frac{\int_{D_r} |f(z)|^2}{\pi r^2}} \\
&= \frac{\sqrt{\int_{D_r} |f(z)|^2}}{\sqrt \pi r}. \label{eq:sup-l2}
\end{align}
In particular, for any $z \in \hat X$ that isn't a root of $\omega$,
we have
\begin{equation} \label{eq:sup-l2-omega}
\left|\frac\beta\omega(z)\right| \le \frac{\norm{\beta}}{\sqrt\pi r_\omega(z)} \le \frac{\norm{\beta}}{r_\omega(z)}.
\end{equation}
where $r_\omega(z)$ is the radius of the largest (open) embedded Euclidean $\omega$-disk around $z$. 
We will also an upper bound of $\beta$ near a root of $\omega$, 
which we will package into the following. 
\begin{lem} \label{lem:int-beta}
Suppose that $z_0$ is an order $n$ root of $\omega$, 
and the $\omega$-disk of radius $2r$ around $z_0$ is embedded, 
without any other singularities of $\omega$.
Then on the $\omega$-disk of radius $r$,
\begin{equation} \label{eq:int-beta1}
\left| \int_{z_0}^z \beta \right|
 \le \norm{\beta}\ln 2(n+1),
\end{equation}
and when $n = 2$,
\begin{equation} \label{eq:int-beta2}
\abs{ \int_{z_0}^z \beta } \le \norm{\beta}.
\end{equation}
\end{lem}
\begin{proof}
We can find a local coordinate $z$ where $z_0$ 
maps to 0, 
and $\omega = (n+1) z^n dz$ in these coordinates. 
In the $z$ coordinate, the $\omega$-disk of radius $2r$ will becomes the disc $D_a$, with $a = (2r)^{1/(n+1)}$.
Similarly the $\omega$-disk of radius $r$ becomes the disc $D_{a'}$, with $a' = r^{1/(n+1)} = 2^{-1/(n+1)} a$. 

For any $z \in D_{a'}$,
we have
$$\beta(z) \le \frac{\norm{\beta}}{a - |z|}$$
by \eqref{eq:sup-l2}.
Therefore, 
for any $z \in D_{a'}$,
\begin{align*}
\left| \int_0^z \beta \right| &\le \norm{\beta} \int_0^{|z|} \frac1 {a - t} dt \\
&\le  \norm{\beta} \int_0^{a'} \frac1 {a - t} dt \\
&= \norm{\beta} \ln \frac a{a - a'} \\
&= \norm{\beta} (-\ln  (1 - 2^{-1/(n+1)}))\\
&\le  \norm{\beta} \ln 2(n+1). 
\end{align*}
In the case where $n = 2$, 
we can use the $\sqrt \pi$ from \eqref{eq:sup-l2} and a direct numerical estimate of $-\ln  (1 - 2^{-1/(n+1)})$
to obtain \eqref{eq:int-beta2}. 
\end{proof}

\begin{thm}
Suppose that the $\omega$-distance between any two zeroes of $\omega$ is at least $2r$. Then
$$\norm{\frac{\betabar}\omega}_{\Teich} \le \frac {4\norm{\beta}} r.$$
\end{thm}
\begin{proof}
Let $\epsilon >0$ be arbitrary. 
For each zero $z_i$ of $\omega$,
we can define a $C^1$ function $p_i \from \hat X \to [0, 1]$ with
\begin{equation} \label{eq:p-r}
\norm{\partial p_i/\omega}_\infty < (2+\epsilon)/r. 
\end{equation}
which is 1 on the disk of radius $r/2$ around $z_i$ and supported in the open disk of radius $r$.
We then define a vector field $v = f/\omega$, 
where $f =  \sum_i p_i \int_{z_i} \betabar$;
we have 
$$\dbar v = \frac1\omega \sum_i \overline{p_i \beta + \partial p_i  \int_{z_i} \beta}.$$
If $q'$ is a norm 1 quadratic differential on $X$
then $q'' \equiv \frac12 \rho_q^* q'$ has double roots at the $z_i$,
so $q''/\omega$ is holomorphic.  
Therefore
\begin{align*}
\int_{\hat X} \dbar v\, q''
= \int_{\hat X} \dbar\! f\, \frac{q''}\omega = 0,
\end{align*}
and hence
\begin{align*}
\norm{\frac{\betabar}\omega}_{\Teich} 
&= \sup_{\norm{q'} = 1} \frac12 \int_{\hat X} q'' \frac{\betabar}\omega \\
&= \sup_{\norm{q'} = 1} \frac12 \int_{\hat X} q''  \left(\frac{\betabar}\omega - \dbar v\right)\\
&\le \norm{\frac{\betabar}\omega - \dbar v}_{\infty}.
\end{align*}
We can bound the latter norm as follows. 
In the disk of radius $r/2$ around $z_i$,
it is zero. 
By \eqref{eq:sup-l2-omega} and \eqref{eq:int-beta2},
in the annulus around $z_i$ with radius between $r/2$ and $r$,
\begin{align*}
\left| \frac{\betabar}\omega - \dbar v \right|  
&= \abs{  \frac{\overline{\beta -p_i \beta-(\partial p_i) \int_{z_i}\beta}}{\omega} } \\
&\le (1-p_i) \abs{ \frac {\betabar}\omega } + \abs{ \frac{\partial p_i}{\omega}} \abs{\int_{z_i}\beta} \\
&\le \frac {2 \norm \beta}r + \frac {(2 + \epsilon)\norm\beta} r = \frac {(4+\epsilon)\norm\beta} r.
\end{align*}
Outside of the disks of radius $r$ around the $z_i$,
$\dbar v$ is zero, and $\abs{\betabar / \omega} < \norm{\beta}/r$.
\end{proof}

\bibliography{HT}{}

\providecommand{\bysame}{\leavevmode\hbox to3em{\hrulefill}\thinspace}
\providecommand{\MR}{\relax\ifhmode\unskip\space\fi MR }
\providecommand{\MRhref}[2]{%
  \href{http://www.ams.org/mathscinet-getitem?mr=#1}{#2}
}
\providecommand{\href}[2]{#2}
\begin{thebibliography}{EMMW20}

\bibitem[ABEM12]{ABEM}
Jayadev Athreya, Alexander Bufetov, Alex Eskin, and Maryam Mirzakhani,
  \emph{Lattice point asymptotics and volume growth on {T}eichm\"{u}ller
  space}, Duke Math. J. \textbf{161} (2012), no.~6, 1055--1111.

\bibitem[AH]{AH3}
Francisco Arana-Herrera, \emph{Effective mapping class group dynamics {III}:
  {C}ounting filling closed curves on surfaces}, forthcoming.

\bibitem[AH20]{AH1}
\bysame, \emph{Effective mapping class group dynamics {I}: {C}ounting lattice
  points in {T}eichm\"uller space}, arXiv:2010.03123 (2020).

\bibitem[AH21]{AH2}
\bysame, \emph{Effective mapping class group dynamics {II}: {G}eometric
  intersection numbers}, arXiv:2104.01694 (2021).

\bibitem[BKN17]{BKN}
Marco Bertola, Dmitry Korotkin, and Chaya Norton, \emph{Symplectic geometry of
  the moduli space of projective structures in homological coordinates},
  Invent. Math. \textbf{210} (2017), no.~3, 759--814.

\bibitem[BMW12]{BMW}
K.~Burns, H.~Masur, and A.~Wilkinson, \emph{The {W}eil-{P}etersson geodesic
  flow is ergodic}, Ann. of Math. (2) \textbf{175} (2012), no.~2, 835--908.

\bibitem[CJY94]{JJ}
Lennart Carleson, Peter~W. Jones, and Jean-Christophe Yoccoz, \emph{Julia and
  {J}ohn}, Bol. Soc. Brasil. Mat. (N.S.) \textbf{25} (1994), no.~1, 1--30.

\bibitem[DH75]{DH}
A.~Douady and J.~Hubbard, \emph{On the density of {S}trebel differentials},
  Invent. Math. \textbf{30} (1975), no.~2, 175--179.

\bibitem[EM18]{EM}
Alex Eskin and Maryam Mirzakhani, \emph{Invariant and stationary measures for
  the {${\rm SL}(2,\Bbb R)$} action on moduli space}, Publ. Math. Inst. Hautes
  \'{E}tudes Sci. \textbf{127} (2018), 95--324.

\bibitem[EMM15]{EMM}
Alex Eskin, Maryam Mirzakhani, and Amir Mohammadi, \emph{Isolation,
  equidistribution, and orbit closures for the {${\rm SL}(2,\Bbb R)$} action on
  moduli space}, Ann. of Math. (2) \textbf{182} (2015), no.~2, 673--721.

\bibitem[EMMW20]{EMMW}
Alex Eskin, Curtis~T. McMullen, Ronen~E. Mukamel, and Alex Wright,
  \emph{Billiards, quadrilaterals, and moduli spaces}, J. Amer. Math. Soc.
  \textbf{33} (2020), no.~4, 1039--1086.

\bibitem[EMR19]{EMR}
Alex Eskin, Maryam Mirzakhani, and Kasra Rafi, \emph{Counting closed geodesics
  in strata}, Invent. Math. \textbf{215} (2019), no.~2, 535--607.

\bibitem[Fil16a]{Fi1}
Simion Filip, \emph{Semisimplicity and rigidity of the {K}ontsevich-{Z}orich
  cocycle}, Invent. Math. \textbf{205} (2016), no.~3, 617--670.

\bibitem[Fil16b]{Fi2}
\bysame, \emph{Splitting mixed {H}odge structures over affine invariant
  manifolds}, Ann. of Math. (2) \textbf{183} (2016), no.~2, 681--713.

\bibitem[FM14]{FM}
Giovanni Forni and Carlos Matheus, \emph{Introduction to {T}eichm\"{u}ller
  theory and its applications to dynamics of interval exchange transformations,
  flows on surfaces and billiards}, J. Mod. Dyn. \textbf{8} (2014), no.~3-4,
  271--436.

\bibitem[For02]{Forni}
Giovanni Forni, \emph{Deviation of ergodic averages for area-preserving flows
  on surfaces of higher genus}, Ann. of Math. (2) \textbf{155} (2002), no.~1,
  1--103.

\bibitem[Fra20]{F}
Ian Frankel, \emph{Meromorphic ${L}^2$ functions on flat surfaces},
  arXiv:2005.13851 (2020).

\bibitem[Gou21]{GS}
\'{E}lize Goujard, \emph{Sous-vari\'{e}t\'{e}s totalement g\'{e}od\'{e}siques
  des espaces de modules de {R}iemann},
  \url{https://www.bourbaki.fr/TEXTES/Exp1178-Goujard.pdf} (2021).

\bibitem[HM79]{HM}
John Hubbard and Howard Masur, \emph{Quadratic differentials and foliations},
  Acta Math. \textbf{142} (1979), no.~3-4, 221--274.

\bibitem[Kon97]{Kont}
M.~Kontsevich, \emph{Lyapunov exponents and {H}odge theory}, The mathematical
  beauty of physics ({S}aclay, 1996), Adv. Ser. Math. Phys., vol.~24, World
  Sci. Publ., River Edge, NJ, 1997, pp.~318--332.

\bibitem[Mas95]{Masur}
Howard Masur, \emph{The {T}eichm\"{u}ller flow is {H}amiltonian}, Proc. Amer.
  Math. Soc. \textbf{123} (1995), no.~12, 3739--3747.

\bibitem[McM13]{McM}
Curtis~T. McMullen, \emph{Navigating moduli space with complex twists}, J. Eur.
  Math. Soc. (JEMS) \textbf{15} (2013), no.~4, 1223--1243.

\bibitem[MMW17]{MMW}
Curtis~T. McMullen, Ronen~E. Mukamel, and Alex Wright, \emph{Cubic curves and
  totally geodesic subvarieties of moduli space}, Ann. of Math. (2)
  \textbf{185} (2017), no.~3, 957--990.

\bibitem[MW18]{MW}
Maryam Mirzakhani and Alex Wright, \emph{Full-rank affine invariant
  submanifolds}, Duke Math. J. \textbf{167} (2018), no.~1, 1--40.

\bibitem[Raf14]{Rafi}
Kasra Rafi, \emph{Hyperbolicity in {T}eichm\"{u}ller space}, Geom. Topol.
  \textbf{18} (2014), no.~5, 3025--3053.

\bibitem[Wol18]{W}
Scott~A. Wolpert, \emph{Schiffer variations and {A}belian differentials}, Adv.
  Math. \textbf{333} (2018), 497--522.

\bibitem[Wri20]{WTG}
Alex Wright, \emph{Totally geodesic submanifolds of {T}eichm\"{u}ller space},
  J. Differential Geom. \textbf{115} (2020), no.~3, 565--575.

\end{thebibliography}
\bibliographystyle{amsalpha}
\end{document}